\documentclass{amsart}

\usepackage{amssymb}
\usepackage{mathtools}

\newtheorem{theorem}{Theorem}[section]
\newtheorem{lemma}[theorem]{Lemma}

\theoremstyle{definition}
\newtheorem{definition}[theorem]{Definition}

\newtheorem{que}{Question}

\theoremstyle{remark}
\newtheorem{remark}[theorem]{Remark}

\numberwithin{equation}{section}

\newcommand{\IN}{\mathbb{N}}
\DeclareMathOperator{\rng}{rng}
\newcommand{\place}{{{}\cdot{}}}

\begin{document}

\title{%
  How to get the Random Graph\\
  with non-uniform probabilities?%
}

\author[L.~Coregliano]{Leonardo N.~Coregliano}
\address[L.~Coregliano]{Institute for Advanced Study, Princeton, NJ, USA}
\email{lenacore@ias.edu}
\thanks{The first-named author's work was supported by a grant from the Institute for Advanced Study School of Mathematics.}

\author{Jaros{\l}aw Swaczyna}
\address[J.~Swaczyna]{Institute of Mathematics, {\L}\'od\'z University of Technology, Aleje Politechniki 8, 93-590 {\L}\'od\'z, Poland}
\email{jaroslaw.swaczyna@p.lodz.pl}

\author{Agnieszka Widz}
\address[A.~Widz]{Institute of Mathematics, {\L}\'od\'z University of Technology, Aleje Politechniki 8, 93-590 {\L}\'od\'z, Poland}
\email{AgnieszkaWidzENFP@gmail.com}
\thanks{The last-named author was supported by the NCN (National Science Centre, Poland), under the Weave-UNISONO call in the Weave programme 2021/03/Y/ST1/00124. The last-named author wants to also express her gratitude towards the Fields Institute in Toronto for organizing semester on Set Theoretic Methods in Algebra, Dynamics and Geometry and providing excellent conditions for research and meeting other mathematicians, in particular first-named author.}

\subjclass[2010]{Primary 05C80, 60C05}

\begin{abstract}
  The Rado Graph, sometimes also known as the (countable) Random Graph, can be generated almost surely by putting an edge between any pair of vertices with some fixed probability $p\in(0,1)$, independently of other pairs.

  In this article, we study the influence of allowing different probabilities for each pair of vertices. More specifically, we characterize for which sequences $(p_n)_{n\in\IN}$ of values in $[0,1]$ there exists a bijection $f$ from pairs of vertices in $\IN$ to $\IN$ such that if we put an edge between $v$ and $w$ with probability $p_{f(\{v,w\})}$, independently of other pairs, then the Random Graph arises almost surely.
\end{abstract}

\maketitle

\section{Introduction}

The Rado Graph is a fascinating object that appears unexpectedly in various areas of Mathematics. First constructed by Ackermann in~\cite{Ack}, it was a matter of interest for Erd\H{o}s and R\'{e}nyi in~\cite{ER}, Rado in~\cite{Rad}, and still attracts many mathematicians, see e.g.~\cite{CDM, DEKK, DM, GT}. The crucial property needed to define the Random Graph is the following.

\begin{definition}
  We say that a graph $(V,E)$ satisfies the property~\ref{star} if
  \begin{equation}\label{star}
    \begin{aligned}
      & \mbox{For all finite disjoint $A, B \subseteq V$ there is a vertex $v\in V$ such that}\\
      & \mbox{$v$ is connected to all elements of $A$ and to no element of $B$.}         
    \end{aligned}
    \tag{$\star$}
  \end{equation}
\end{definition}

This definition has three immediate consequences: a simple induction shows that~\ref{star} in fact implies that there are infinitely many $v$ connected to $A$ and to no element of $B$, any graph satisfying~\ref{star} must be infinite and have infinitely many edges and non-edges, and any two countable graphs satisfying property~\ref{star} are isomorphic (this follows by a standard back-and-forth argument). The latter observation allows us to call any countable graph with~\ref{star} the \emph{Random Graph}.

Cameron, is his paper~\cite{Ca}, presented a nice introduction to the topic, providing a number of instances where the Random Graph appears and explaining some of its basic properties. Therefore, we refer the reader to this paper for more detailed introduction. In the presented note we want to discuss some issues related to one of the most standard constructions leading to the Random Graph. Therefore, we will now sketch this construction and discuss some of its aspects. 

The simplest, although not exactly explicit, way of generating the Random Graph is by fixing a countably infinite set $V$ and declaring that any pair of vertices $\{v,w\}$ to be an edge with probability exactly $1/2$, independently of other pairs. It is straightforward to verify that with probability $1$, the resulting graph will satisfy property~\ref{star}, making it the Random Graph (in fact this is a consequence of the fact that in infinite coin flip any finite sequence of tails and heads appears infinitely many times). Putting it more simply, if $G_{\IN,1/2}$ is the countable Erd\H{o}s--R\'{e}nyi random graph model, then $G_{\IN,1/2}$ is isomorphic to the Random Graph with probability $1$. Now one may wonder, if there is something special in the probability $1/2$ used in this construction. In other words, we ask the following question about a property of a sequence of probabilities.

\begin{que}\label{Q0}
  For which assignments of probabilities to the edges do we obtain the Random Graph with probability $1$?\footnote{This question arose during the second-named author's collaborative work on the Random Graph with his bachelor's student, Aleksandra Czerczak.}
\end{que}

Even though the above question looks very natural, and the Random Graph was introduced in the first half of the 20th century, we were unable to find direct answer in the existing literature. Therefore, the aim of the presented note is to give an answer and also to popularize the fascinating object, that the Random Graph is, among wider audience. Another remark is that the question above it is not very precise, but now we will discuss it in order to formulate the right one. It is easy to see that if we replace $1/2$ by any other probability $p\in(0,1)$ we still get the property~\ref{star} (i.e., $G_{\IN,p}$ is also almost surely isomorphic to the Random Graph). But what happens if we allow different probabilities for various edges? An initial observation here is that if these probabilities are separated from $0$ and $1$, then we still get the Random Graph. 

It is natural to consider the case of probabilities tending to $0$ (or $1$) now, but we have to clarify some subtleties before this. Namely, note that we assign a probability to each pair of vertices; thus, formally, we do not have the sequence of probabilities. Of course we may rearrange them to get a sequence, but this idea requires some extra caution. Note that the property of generating the Random Graph is not invariant with respect to permutations! Indeed, suppose that we have some fixed arrangement of probabilities that generates the Random Graph, but the probabilities are not separated from $0$ (the second case is completely analogous). Then for every $\varepsilon>0$ we may split the probabilities into two infinite sets, say $C$ and $D$, such that the sum of elements of $C$ is smaller than $\varepsilon$, and $D$ contains the rest of them. Now fix arbitrary vertex $v\in V$ and assign the probabilities in such a way that elements of $D$ are probabilities of those edges, for which $v$ is one of the ends, and probabilities of all other edges are elements of $C$. Note that in such a case, probability of the existence of any edge for which $v$ is not an endpoint is less than $\varepsilon$; hence, with probability $1$ property~\ref{star} will not be satisfied. Therefore, right thing to do is considering properties of the sequence of probabilities, rather than some particular assignment, and the precise way to formulate the Question~\ref{Q0} is the following one.

\begin{que} \label{QM}
  Let $V$ be a countably infinite set. For which sequences $(p_n)_{n\in \IN}$ of elements of the interval $[0,1]$ there exists a bijection $f\colon [V]^2 \to \IN$ such that if we set probability of existence of edge $\{v,w\}$ as $p_{f(\{v,w\})}$ (to be picked independently from other pairs), with probability $1$ the resulting graph will be the Random Graph? 
\end{que} 

Let us conclude the introduction with another easy observation. Namely, suppose that the sequence of probabilities $(p_n)_{n\in\IN}$ has a finite series $\sum_{n=0}^\infty p_n < \infty$. Then, for any assignment of $p_n$'s values to edges between vertices from $V$, one won't get the Random Graph with probability $1$. Indeed, the (first) Borel--Cantelli Lemma, implies that with probability $1$, the graph will have only finitely many edges, hence~\ref{star} does not hold. As we will see, a similar almost sure finiteness argument is the only obstacle for producing the Random Graph.

\section{Preliminaries}

We will denote by $\IN$ the set of non-negative integers and for a set $V$ and $k\in\IN$, we denote by $[V]^k$ the set of subsets of $V$ of cardinality exactly $k$.

We start with a few lemmas that will be needed to prove our main theorem. 
\begin{lemma}\label{LemMult}
  Let $(a_n)_{n\in\IN}$ be a non-increasing sequence of elements of interval $[0,1]$ and let $k \in \IN$. Suppose that $\lim_{n\to\infty} a_n =0$ and $\sum_{n=0}^\infty a_n^k = \infty$. Then 
  \begin{align*}
    \sum_{m=0}^\infty \prod_{i=0}^{k-1} a_{mk+i} & = \infty.
  \end{align*}
\end{lemma}
\begin{proof}
  Let $b_{mk+i}=a_{mk}$ for $m \in \IN$, $i \in \{0, \ldots, k-1\}$ and note that 
  \begin{align*}
    \infty & =
    \sum_{n=0}^\infty a_n^k \leqslant
    \sum_{n=0}^\infty b_n^k = \sum_{m=0}^\infty k a_{mk}^k = k\sum_{m=0}^\infty a_{mk}^k,  
  \end{align*}
  so by omitting the first term of the last sum, we conclude that $\sum_{n=1}^\infty a_{nk}^k= \infty$. Therefore,
  \begin{align*}
    \sum_{m =0}^\infty \prod_{i=0}^{k-1} a_{mk+i}  & \geqslant \sum_{m =0}^\infty \prod_{i=0}^{k-1} b_{(m+1)k+i} = \sum_{m =0}^\infty \prod_{i=0}^{k-1} a_{(m+1)k} = \sum_{n =1}^\infty a_{nk}^k= \infty.
    \qedhere
  \end{align*}
\end{proof}

\begin{lemma}\label{LemR}
  Let $(a_n)_{n\in\IN}$ be a sequence of elements in the interval $[0,1]$. Suppose that $\lim_{n\to\infty} a_n =0$ and $\sum_{n=0}^\infty a_n^k = \infty$ for every $k \in \IN$. Then $(a_n)_{n\in\IN}$ may be split into infinitely many subsequences $((a_{n^{i,k}_\ell})_\ell)_{i,k}$ such that for every $i,k \in\IN$ $\sum_{\ell=0}^\infty a_{n^{i,k}_\ell}^k=\infty$.
\end{lemma}

\begin{proof}
  Fix an enumeration $(i_m,k_m)_{m\in\IN}$ of pairs $(i,k) \in \IN$ such that each pair $(i,k)$ appears infinitely many times. Then find $n_0 \in \IN$ such that $\sum_{n=0}^{n_0} a_n^{k_0} \geqslant 1$ and assign elements $a_0, \ldots, a_{n_0}$ to the sequence $(a_{n^{i_0,k_0}_\ell})_\ell$. Next find an $n_1$ such that $\sum_{n=n_0 + 1}^{n_1} a_n^{k_1} \geqslant 1$, and assign elements $a_{n_0+1}, \ldots, a_{n_1}$ to the sequence $(a_{n^{i_1,k_1}_\ell})_\ell$. Proceeding inductively, we satisfy our claim.
\end{proof}

\begin{lemma}\label{LemM}
  If for each natural number $k$ the sums
  $\sum_{n=0}^\infty p_n^k$ and $\sum_{n=0}^\infty (1- p_n)^k$ are infinite, 
  then there exists an injection $f\colon \mathfrak{X} \to \IN$ such that 
  \begin{align*}
    \forall k, m \in \IN,
    \sum_{n=0}^\infty \prod_{i=0}^{ k-1} p_{f(k, m, n, i)}\prod_{i=k}^{ 2k-1} (1-p_{f(k, m, n, i)})=\infty,
  \end{align*}
  and $\IN \setminus \rng(f)$ is infinite, where
  \begin{align*}
    \mathfrak{X} & \coloneqq \{(k,m,n,i)\in\IN^4 \mid i\leq 2k-1\}.
  \end{align*}
\end{lemma}

\begin{proof}
  Consider the following cases. 
  
  In the first case, there exists $\varepsilon>0$ such that the set $M_\varepsilon\coloneqq\{n \in \IN: \varepsilon \leqslant p_n\leqslant 1-\varepsilon\}$ is infinite. Then any injection $f \colon \mathfrak{X} \to M_\varepsilon$ with coinfinite range works. Indeed, note that for each $(k, m, n, i) \in \mathfrak{X}$ we have $p_{f(k, m, n, i)}, 1-p_{f(k, m, n, i)} \geqslant \varepsilon.$ Therefore, all terms of the considered sum are at least $\varepsilon^{2k}$, hence the sum is infinite. 

  In the second case, for every $\varepsilon>0$ the set $M_\varepsilon$ defined above is finite, but both sets $\{n \in \IN: p_n\leqslant \varepsilon\}$, $\{n \in \IN: p_n\geqslant 1-\varepsilon\}$ are infinite. Then we may fix a partition $\IN=A\cup B$ such that $(p_n)_{n\in A}$ converges to $0$ and $(p_n)_{n\in B}$ converges to $1$. For $(k,m,n,i) \in \mathfrak{X}$ put $f(k,m,n,i)\in A$ if $i \geqslant k $ and $f(k,m,n,i)\in B$ if $i \leqslant k-1$, while ensuring that $\IN \setminus \rng(f)$ is infinite. Then for all but finitely many $(k,m,n,i) \in \mathfrak{X}$ we have $p_{f(k,m,n,i)}, 1- p_{f(k,m,n,i)} > 1/2$, hence the considered sum is infinite. 
  
  In the final case, either $p_n \to 0$ or $p_n \to 1$. We will deal only with the first one, since the second one is analogous. By passing to a subsequence, we may assume that all $p_n$'s are positive. Let us use Lemma~\ref{LemR} to split $(p_n)_{n\in\IN}$ into $((p_{s^{m,k}_\ell})_\ell)_{m,k}$ such that for every $m,k \in\IN$, we have $\sum_{\ell=0}^\infty p_{s^{m,k}_\ell}^k=\infty$. Without loss of generality, we may assume that for every $m,k \in\IN$ the sequence $(p_{s^{m,k}_\ell})_\ell$ is non-increasing. Now, set $f(k,m,n,i)=s^{m+1,k}_{nk+i}$ for $i \in \{0, \ldots, k-1\}$ ($+1$ is just to leave infinitely many elements unused). Since $p_n \to 0$ we may set $f(k,m,n,i)$ for $i \in \{k, \ldots, 2k-1\}$ such that range of $f$ is co-infinite and $p_{f(k,m,n,i)}<1/2$. Then Lemma~\ref{LemMult} yields
  \begin{align*}
    \sum_{n=0}^\infty \prod_{i=0}^{ k-1} p_{f(k, m, n, i)} \prod_{i=k}^{2k-1} (1-p_{f(k, m, n, i)})
    & \geqslant
    \sum_{n=0}^\infty \prod_{i=0}^{k-1} p_{f(k, m, n, i)} \prod_{i=k}^{2k-1} \frac{1}{2}
    \\
    & = \left(\frac{1}{2}\right)^k \sum_{n=0}^\infty \prod_{i=0}^{k-1} p_{s_{nk+i}^{m+1,k}} 
    = \infty.
    \qedhere
  \end{align*}
\end{proof}

\section{Main theorem}

In this section, we formulate and prove the main theorem of this note, which fully answers Question~\ref{QM}. In fact, our theorem shows a $0/1$-law regarding the problem: either there exists a bijective assignment that generates the Random Graph with probability $1$, or for every bijective assignment the Random Graph is generated with probability $0$.
\begin{theorem}\label{mainthm}
  The following are equivalent for a sequence $(p_n)_{n \in \IN}$ of numbers from interval $[0,1]$.
  \begin{enumerate}
  \item\label{RGP1} There exists a bijective assignment $f\colon [\IN]^2\to\IN$ such that by letting
    each $\{v,w\}\in[\IN]^2$ be an edge with probability $p_{f(\{v,w\})}$, independently from other
    pairs, the resulting graph is the Random Graph with probability $1$.
  \item\label{RGPpos} Item~\eqref{RGP1} holds but the conclusion holds with positive probability
    instead of probability $1$.
  \item\label{series} For every $k\in\IN$, the sums $\sum_{n=0}^\infty p_n^k$ and
    $\sum_{n=0}^\infty (1-p_n)^k$ are infinite.
  \end{enumerate}
\end{theorem}

Before we prove the theorem, let us note that a standard example of a sequence $(p_n)_{n\in\IN}$
satisfying item~\eqref{series} that is not bounded away from $0$ is $p_n\coloneqq 1/(\log(n+3))$.

\begin{proof}
  We will first deal with harder implication~\eqref{series}$\implies$\eqref{RGP1}, namely, we will construct the proper assignment of probabilities provided that $\sum_{n=0}^\infty p_n^k=\sum_{n=0}^\infty (1-p_n)^k=\infty$ for every $k\in\IN$. Note that in order to check property~\ref{star} it is enough to consider sets $A, B$ of the same cardinality (by possibly taking a superset of the smaller set, disjoint from the other one). Let us then enumerate all pairs of finite disjoint subsets of $\IN$ of same size as $(A_n,B_n)_{n\in\IN}$.

  Let $f \colon \mathfrak{X} \to \IN$ be provided by Lemma~\ref{LemM}. For every $n \in \IN$, let
  $k_n\coloneqq\lvert A_n\rvert=\lvert B_n\rvert$ and define inductively in $n$ sets $C_n$ and $D_n$
  as follows: set $D_{-1}\coloneqq\IN$ and for each $n\in\IN$, let $C_n\subseteq D_{n-1}\setminus
  (A_n\cup B_n)$ be an infinite set with $D_n\coloneqq D_{n-1}\setminus(A_n\cup B_n\cup C_n)$ also
  infinite.

  Note that this definition inductively ensures that for every $\{v,w\}\in[\IN]^2$, there exists at most one $n\in\IN$ such that $\{v,w\}$ intersects both $A_n\cup B_n$ and $C_n$ in exactly one point each. This means that we can proceed inductively to determine the probabilities of existence of edges between $A_n\cup B_n$ and $C_n$ in the $n$'th step of our induction without running the risk of defining $p_{\{v,w\}}$ more than once.

  When handling $(A_n\cup B_n,C_n)$ in the $n$'th step of induction, we will use the first unused sequence given by $f$ that is suitable for $k_n=\lvert A_n\rvert=\lvert B_n\rvert$. Formally, we set $\ell_n\coloneqq\lvert\{i<n: k_i=k_n\}\rvert$ and use the probabilities $p_{f(k_n, \ell_n, \place, \place)}$. More precisely, let us enumerate $A_n = \{a_0<a_1 < \cdots < a_{k-1}\}$, $B_n = \{b_0<b_1 < \cdots < b_{k-1}\}$, $C_n = \{c_0<c_1 < \cdots \}$, and set the probability of existence of the edge between $a_i$ and $c_j$ as $p_{f(k_n, \ell_n,j,i)}$, while the probability of existence of the edge between $b_i$ and $c_j$ we set as $p_{f(k_n, \ell_n,j,i+k_n)}$. Finally, as Lemma~\ref{LemM} leaves us with infinitely many unused probabilities, we assign them to the missing edges. Note that there are infinitely many such edges, as we have e.g.\ left all edges between elements of $C_0 \setminus (A_1\cup B_1)$ and $C_1$ free. 

  Let us check that the given construction produces the Random Graph with probability $1$. Indeed, let us fix finite disjoint sets with the same cardinality $A, B \subseteq \IN$. We have to check that with probability $1$ there is a vertex $v$ connected to all elements of $A$ and to no element of $B$; let us call this property ``being well-connected to $(A,B)$''. Let $n \in \IN$ be such that $(A, B)=(A_n,B_n)$ and note that for a fixed element $c_j \in C_n$, the probability that $c_j$ is well-connected to $(A,B)$ is exactly 
  \begin{align*}
    \prod_{i=0}^{k-1} p_{f(k_n, \ell_n, j, i)} \cdot \prod_{i=k}^{2k-1} \left( 1- p_{f(k_n, \ell_n, j, i)}\right),
  \end{align*}
  so, by Lemma~\ref{LemM}, the sum of those probabilities over all $c_j$'s is infinite. Note that if $j \neq j'$, then well-connectedness of $c_j$ and $c_{j'}$ to $(A,B)$ are clearly independent. Therefore, by the Second Borel--Cantelli Lemma, with probability $1$ there exists infinitely many $c_j$'s that are well-connected to $(A,B)$. Since there are countably many pairs $(A,B)$ we conclude that with probability $1$ property~\ref{star} is satisfied.

  \medskip

  The implication~\eqref{RGP1}$\implies$\eqref{RGPpos} is obvious, so it remains to prove the implication~\eqref{RGPpos}$\implies$\eqref{series}, which we prove by its contra-positive: we will show that if there exists $k\in\IN$ such that either $\sum_{n=0}^\infty p_n^k$ or $\sum_{n=0}^\infty (1-p_n)^k$ is finite, then with probability $1$, the resulting graph is \emph{not} the Random Graph.

  We prove only the case when $\sum_{n=0}^\infty p_n^k$ is finite as the other case is analogous. Let $A$ be any set of cardinality $k$, enumerate its elements as $a_0,\ldots,a_{k-1}$ and the elements of $\IN\setminus A$ as $v_0,v_1,\ldots$. For each $m\in\IN$ and each $i\leq k-1$, let $p_{n_{m,i}}\in[0,1]$ be the probability value assigned to $\{a_i,v_m\}$. Note that
  \begin{align*}
    \sum_{m=0}^\infty \prod_{i=0}^{k-1} p_{n_{m,i}}
    & \leq
    \sum_{m=0}^\infty \max_{i\leq k-1} p_{n_{m,i}}^k
    \leq
    \sum_{m=0}^\infty \sum_{i=0}^{k-1} p_{n_{m,i}}^k
    \leq
    \sum_{n=0}^\infty p_n^k
    <
    \infty,
  \end{align*}
  so by the (first) Borel--Cantelli Lemma, it follows that with probability $1$, there are only finitely many $v_j$ that are adjacent to all vertices of $A$. Therefore, by adding to the set $A$ those finitely many vertices, we see that~\ref{star} does not hold.
\end{proof}

Note that the result of this article easily extends to the Random $t$-Hypergraph. Namely, for $t\geq
2$, we say that a $t$-hypergraph $(V,E)$ has the property~\ref{tstar} if
\begin{equation}\label{tstar}
  \begin{aligned}
    & \mbox{For all finite disjoint $A, B\subseteq [V]^{t-1}$ there is a vertex $v\in V$ such that}\\
    & \mbox{$a\cup\{v\}\in E$ for every $a\in A$ and $b\cup\{v\}\notin E$ for every $b\in B$.}
  \end{aligned}
  \tag{$\star_t$}
\end{equation}
Again a simple back-and-forth argument shows that there is a unique (up to isomorphism) countable
$t$-hypergraph with property~\ref{tstar}, which we call the \emph{Random $t$-Hypergraph} and a
simple way of generating the Random $t$-Hypergraph with probability $1$ is to declare each $t$-set
to be an edge with probability $1/2$, independently from other $t$-sets. Finally, the following result analogous
to Theorem~\ref{mainthm} holds for the Random $t$-Hypergraph with an analogous proof:
\begin{theorem}
  The following are equivalent for $t\geq 2$ and a sequence $(p_n)_{n \in \IN}$ of numbers from
  interval $[0,1]$.
  \begin{enumerate}
  \item\label{RHP1} There exists a bijective assignment $f\colon [\IN]^t\to\IN$ such that by letting
    each $e\in[\IN]^t$ be an edge with probability $p_{f(e)}$, independently from other $t$-sets,
    the resulting $t$-hypergraph is the Random $t$-Hypergraph with probability $1$.
  \item Item~\eqref{RHP1} holds but the conclusion holds with positive probability instead of
    probability $1$.
  \item For every $k\in\IN$, the sums $\sum_{n=0}^\infty p_n^k$ and
    $\sum_{n=0}^\infty (1-p_n)^k$ are infinite.
  \end{enumerate}
\end{theorem}
\begin{remark}
Notice an unexpected resemblence of our Theorem with \cite[Theorem 1.3]{Bart}, where Bartoszy\'{n}ski tries to characterize for which measures $\mu$ on $2^\IN$ all filters on $\IN$ are $\mu$-measureable. 
\end{remark}

\subsection*{Acknowledgement}
The authors are gratefull to S\l awomir Solecki for inspiring discussions.

\bibliographystyle{amsplain}

\begin{thebibliography}{10}
\bibitem{Ack} Ackermann, W. (1937), "Die Widerspruchsfreiheit der allgemeinen Mengenlehre", \emph{Mathematische Annalen}, 114 (1): 305--315

\bibitem{Bart}  T. Bartoszy\'{n}ski (1991/92), "On the structure of measureable filters on a countable set" \emph{Real Anal. Exchange}, 17 (2), 681 -- 701, 

\bibitem{Ca} Cameron P. J.(2013). "The Random Graph",
  \emph{The Methematics of Paul Erd\H os II}.
  Nowy Jork: Springer. ISBN 978-1-4614-7254-4

\bibitem{CDM} Chatterjee S., Diaconis P., Miclo L. (2022). "A random walk on the {R}ado graph",
  \emph{Toeplitz operators and random matrices---in memory of {H}arold {W}idom}.
  Cham: Birkh\"{a}user/Springer

\bibitem{DEKK} Darji U. B., Elekes M., Kalina K., Kiss V., Vidnyánszky Z., (2022). "The structure of random automorphisms of the random graph",
  \emph{Annals of Pure and Applied Logic 173, \textbf{9}}.

\bibitem{DM} Diaconis P., Malliaris M. (2021). "Complexity and randomness in the {H}eisenberg groups (and beyond)", \emph{New Zealand Journal of Mathematics 52}.

\bibitem{ER} Erd\H{o}s, P.; R\'{e}nyi, A. (1963), "Asymmetric graphs", \emph{Acta Math. Acad. Scien. Hung.}, \textbf{14} (3-4): 295--315

\bibitem{GT} Guzm\'{a}n O., Todorcevic S. (2023). "Forcing with copies of the {R}ado and {H}enson graphs",
  \emph{Annals of Pure and Applied Logic 174, nr 8}.

\bibitem{Rad} Rado, Richard (1964), "Universal graphs and universal functions", \emph{Acta Arith.}, \textbf{9} (4): 331--340, 
\end{thebibliography}

\end{document}